\documentclass[12pt, reqno]{amsart}
\usepackage{amssymb, latexsym, amsmath, amsthm}
\usepackage[bookmarksnumbered, colorlinks, plainpages]{hyperref}
\usepackage{tikz-cd}
\usepackage[all]{xy}
\usepackage{amsmath}
\usepackage{tikz-cd}
\newtheorem{theorem}{Theorem}[section]
\newtheorem{lemma}[theorem]{Lemma}
\newtheorem{proposition}[theorem]{Proposition}
\newtheorem{corollary}[theorem]{Corollary}
\theoremstyle{definition}
\newtheorem{definition}[theorem]{Definition}

\newtheorem{remarks}[theorem]{Remarks}
\numberwithin{equation}{section}

\newcommand{\norma}[1]{\| #1 \|}
\newcommand {\N} {\mathbb{N}}

\newcommand{\sol}[1]{\text{sol}(#1)}
\newcommand{\conj}[2]{\left \{ {#1} \, : \, {#2} \right \}}

\begin{document}
	\setcounter{page}{1}
	
	\title{Gelfand-Phillips-type properties in Banach lattices}
	
	\author[Ardakani and Miranda]
	{Halimeh Ardakani$^{*}$ and VINÍCIUS C. C. MIRANDA$\dagger$}
	
	\address {$^{1}$ Department of Mathematics, Payame Noor University, P. O. Box 19395-3697, Tehran, Iran.}
	
	\address {$^{2}$ Centro de Matem\'atica, Computa\c c\~ao e Cogni\c c\~ao, Universidade Federal do ABC, 09210-580, Santo Andr\'e,
		Brazil.}
	\email{\textcolor[rgb]{0.00,0.00,0.84}{ardakani@pnu.ac.ir, halimeh\_ardakani@yahoo.com}}
	\email{\textcolor[rgb]{0.00,0.00,0.84}{colferaiv@gmail.com}}

	
	
	\keywords{$p$-limited, almost $p$-limited set, $p$-Gelfand-Phillips property,  $p$-limited completely continuous operator, discrete Banach lattice.
		\newline \indent $^{*}$ Corresponding author.
		\newline \indent $^\dagger$Supported by FAPESP grant 2023/12916-1 and Fapemig grant APQ-01853-23}
	\begin{abstract}

        We study $p$-limited and almost $p$-limited sets in Banach lattices and their connections with relatively $p$-compact and relatively compact sets. We investigate the weak and the strong 
        Gelfand-Phillips property of order $p$, as well as the $p$-GP property introduced by Delgado and Piñeiro, providing conditions under which these properties may coincide. Additionally, we prove that a Banach lattice $E$ is a KB-space if and only if every almost $p$-limited set in $E$ is 
     relatively weakly compact if and only if every the adjoint of a weakly compact taking values on $E$ is disjoint $p$-summing.
	\end{abstract} \maketitle

\section{Introduction}

Recall that a subset $A$ of a Banach space $X$ is said to be \textit{limited} if every weak*-null sequence in the dual space $X^*$ converges uniformly to $0$ on $A.$ Or equivalently, if for every bounded linear operator $T:X\to c_0,$ $T(A)$ is a relatively compact set. Notions related to limitedness have been considered in the mathematical literature. For instance, in the context of summability properties, the so called $p$-\textsl{limited} ($1 \leq p < \infty$) sets were defined by Karn and Sinha \cite{karnsinha2} and further studied by Delgado and Pi\~neiro \cite{delpin2}, Galindo and Miranda \cite{coarse, galmirams}, and Ghenciu \cite{Ioana, p-lcc}.
We recall the definition of $p$-limited sets: 
a subset $A$ of a Banach space $X$ is a 
    $p$-limited set ($1 \leq p < \infty$) if for every weak* $p$-summable sequence $(x_n^*)_n \subset X^*$, there is $a = (a_n)_n \in \ell_p$ such that $|x_n^*(x)| \leq a_n$ for every $x \in A$ and $n \in \N$. Or equivalently, if for every bounded linear operator $T: X \to \ell_p$, $T(A)$ is an order bounded subset of $\ell_p$.

These two classes of sets play an important role in the study of certain properties of Banach spaces and in the theory of bounded linear operators. For example, the Gelfand-Phillips property (resp., Gelfand-Phillips property of order $p$) arises from comparing limited sets with relatively compact sets (resp., $p$-limited sets with relatively $p$-compact sets). In \cite{p-lcc}, I. Ghenciu introduced the weak Gelfand-Phillips property of order $p$ in a Banach space by comparing $p$-limited sets with relatively compact sets. The study of this property constitutes the main subject of Section \ref{Secweakpgp} of the present manuscript. It is worth mentioning that while every Banach space with the Gelfand-Phillips property of order $p$ necessarily has the weak Gelfand-Phillips property of order $p$, 
it remains an open problem whether there exists a Banach space with the weak Gelfand-Phillips property of order $p$ that fails to have the  Gelfand-Phillips property of order $p$. In Section \ref{Secweakpgp}, we provide sufficient conditions in the setting of Banach lattices under which these two properties are equivalent. For instance, we prove that if a Banach lattice solves this problem, then it must be a non-discrete KB-space that is neither reflexive nor an AL-space. Along the way, we obtain several interesting results, such as Theorem \ref{p-limited} which states that every $p$-limited subset of a discrete KB-space must be relatively $p$-compact, and Corollary \ref{fShj}, which provides a characterization of when the dual of a Banach space has the weak Gelfand-Phillips property of order $p$ in terms of the original space not containing a copy of $\ell_1$. The definitions and examples related to the Gelfand-Phillips properties of order $p$ will be presented at the beginning of Section \ref{Secweakpgp}.

In Section \ref{secstrong}, we introduce the
strong Gelfand Phillips property of order $p$ in a Banach lattice by considering the almost $p$-limited sets introduced in \cite{galmirams}, which we recall the definition: a subset $A$ of a Banach lattice $E$ is \textit{almost $p$-limited} (resp. \textit{positive $p$-limited}) if for each disjoint (resp. positive)  weak$^*$ $p$-summable sequence $(x_n^*)_n \subset E^*$, there exists $(a_j)_j \in \ell_p$ such that $|x_n^*(x)| \leq a_n$ for every $x \in A$ and $n \in \N$. By a \textit{disjoint operator} $T: E \to \ell_p$, we mean that the weak$^*$ $p$-summable sequence $(x_n^*)_n \subset E^*$ that defines $T$, that is $T(x) = (x_n^*(x))_n$ for every $x \in E$, is disjoint. It was noted in \cite{galmirams} that, a subset $A \subset E$ is almost $p$-limited (resp. positive $p$-limited) if and only if $T(A)$ is an order bounded subset of $\ell_p$ for every disjoint (resp. positive) operator $T: E \to \ell_p$. We then give characterizations and examples for this new property. Moreover, we also 
provide conditions under which all Gelfand-Phillips properties of order $p$ mentioned above coincide. In particular, the following chain of implications concerning the Gelfand-Phillips type properties hold for all $p \geq 2$:

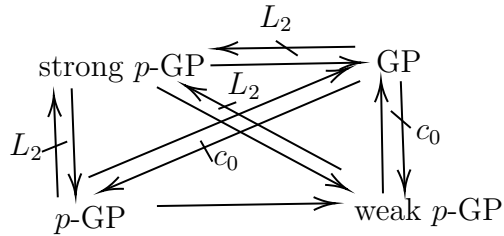
\begin{center}

\tikzset{every picture/.style={line width=0.75pt}} 

\begin{tikzpicture}[x=0.75pt,y=0.75pt,yscale=-1,xscale=1]
	
	\draw  (210,160) -- (284,158.64);
	\draw [shift={(286,158.6)}, rotate = 178.94] [color={rgb, 255:red, 0; green, 0; blue, 0 }   ] [line width=0.75]    (10.93,-3.29) .. controls (6.95,-1.4) and (3.31,-0.3) .. (0,0) .. controls (3.31,0.3) and (6.95,1.4) .. (10.93,3.29)   ;
	\draw  (140,171.6) -- (141.93,223.6) ;
	\draw [shift={(142,225.6)}, rotate = 267.88] [color={rgb, 255:red, 0; green, 0; blue, 0 }  ,draw opacity=1 ][line width=0.75]    (10.93,-3.29) .. controls (6.95,-1.4) and (3.31,-0.3) .. (0,0) .. controls (3.31,0.3) and (6.95,1.4) .. (10.93,3.29)   ;
	\draw    (306,168) -- (307.93,221.6) ;
	\draw [shift={(308,223.6)}, rotate = 267.94] [color={rgb, 255:red, 0; green, 0; blue, 0 }  ,draw opacity=1 ][line width=0.75]    (10.93,-3.29) .. controls (6.95,-1.4) and (3.31,-0.3) .. (0,0) .. controls (3.31,0.3) and (6.95,1.4) .. (10.93,3.29)   ;
	\draw   (283,150.6) -- (215,151.57) ;
	\draw [shift={(213,151.6)}, rotate = 359.18] [color={rgb, 255:red, 0; green, 0; blue, 0 }  ,draw opacity=1 ][line width=0.75]    (10.93,-3.29) .. controls (6.95,-1.4) and (3.31,-0.3) .. (0,0) .. controls (3.31,0.3) and (6.95,1.4) .. (10.93,3.29)   ;
	\draw   (243,147.6) -- (254,155.1) ;
	\draw    (133,226.6) -- (131.08,177) ;
	\draw [shift={(131,175)}, rotate = 87.78] [color={rgb, 255:red, 0; green, 0; blue, 0 }  ,draw opacity=1 ][line width=0.75]    (10.93,-3.29) .. controls (6.95,-1.4) and (3.31,-0.3) .. (0,0) .. controls (3.31,0.3) and (6.95,1.4) .. (10.93,3.29)   ;
	\draw   (125,203.6) -- (137.75,198.45) ;
	\draw    (297,224.6) -- (296.04,169.6) ;
	\draw [shift={(296,167.6)}, rotate = 88.99] [color={rgb, 255:red, 0; green, 0; blue, 0 }  ,draw opacity=1 ][line width=0.75]    (10.93,-3.29) .. controls (6.95,-1.4) and (3.31,-0.3) .. (0,0) .. controls (3.31,0.3) and (6.95,1.4) .. (10.93,3.29)   ;
	\draw    (301,183) -- (311,191.6) ;
	\draw   (183,231) -- (271,229.63) ;
	\draw [shift={(273,229.6)}, rotate = 179.11] [color={rgb, 255:red, 0; green, 0; blue, 0}  ,draw opacity=1 ][line width=0.75]    (10.93,-3.29) .. controls (6.95,-1.4) and (3.31,-0.3) .. (0,0) .. controls (3.31,0.3) and (6.95,1.4) .. (10.93,3.29)   ;
	\draw  (183,171) -- (277.24,221.65) ;
	\draw [shift={(279,222.6)}, rotate = 208.26] [color={rgb, 255:red, 0; green, 0; blue, 0 }  ,draw opacity=1 ][line width=0.75]    (10.93,-3.29) .. controls (6.95,-1.4) and (3.31,-0.3) .. (0,0) .. controls (3.31,0.3) and (6.95,1.4) .. (10.93,3.29)   ;
	\draw   (276,212.6) -- (198.76,170.95) ;
	\draw [shift={(197,170)}, rotate = 28.34] [color={rgb, 255:red, 0; green, 0; blue, 0}  ,draw opacity=1 ][line width=0.75]    (10.93,-3.29) .. controls (6.95,-1.4) and (3.31,-0.3) .. (0,0) .. controls (3.31,0.3) and (6.95,1.4) .. (10.93,3.29)   ;
	\draw   (214,182) -- (223,178.6) ;
	\draw    (148.5,216.5) -- (277.66,162.27) ;
	\draw [shift={(279.5,161.5)}, rotate = 157.23] [color={rgb, 255:red, 0; green, 0; blue, 0}  ,draw opacity=1 ][line width=0.75]    (10.93,-3.29) .. controls (6.95,-1.4) and (3.31,-0.3) .. (0,0) .. controls (3.31,0.3) and (6.95,1.4) .. (10.93,3.29)   ;
	\draw    (285.5,167.9) -- (156.84,222.81) ;
	\draw [shift={(155,223.6)}, rotate = 336.89] [color={rgb, 255:red, 0; green, 0; blue, 0 }  ,draw opacity=1 ][line width=0.75]    (10.93,-3.29) .. controls (6.95,-1.4) and (3.31,-0.3) .. (0,0) .. controls (3.31,0.3) and (6.95,1.4) .. (10.93,3.29)   ;
	\draw   (205,198) -- (211,208.6) ;
	
	\draw (122,153) node [anchor=north west][inner sep=0.75pt]   [align=left] {strong $p$-GP};
	\draw (292,151) node [anchor=north west][inner sep=0.75pt]   [align=left] {GP};
	\draw (130,229.6) node [anchor=north west][inner sep=0.75pt]   [align=left] {$p$-GP};
	\draw (233,128) node [anchor=north west][inner sep=0.75pt]  [color={rgb, 255:red, 0; green, 0; blue, 0 }  ,opacity=1 ] [align=left] {$ L_{2}$};
	\draw (312,189) node [anchor=north west][inner sep=0.75pt]  [color={rgb, 255:red, 0; green, 0; blue, 0 }  ,opacity=1 ] [align=left] {$ c_{0}$};
	\draw (107,193) node [anchor=north west][inner sep=0.75pt]  [color={rgb, 255:red, 0; green, 0; blue, 0 }  ,opacity=1 ] [align=left] {$L_{2}$};
	\draw (281,225.6) node [anchor=north west][inner sep=0.75pt]   [align=left] {weak $p$-GP};
	\draw (217,162) node [anchor=north west][inner sep=0.75pt]  [color={rgb, 255:red, 0; green, 0; blue, 0}  ,opacity=1 ] [align=left] {$\displaystyle L_{2}$};
	\draw (210,202) node [anchor=north west][inner sep=0.75pt]  [color={rgb, 255:red, 0; green, 0; blue, 0 }  ,opacity=1 ] [align=left] {$\displaystyle c_{0}$};
\end{tikzpicture}
\end{center}
In the above diagram, the cut arrows mean that the implication does not hold and the counterexample is the set within the arrow. Also, $L_2$ means $L_2([0,1])$.

Another interesting question we address is the characterization of a Banach lattice in which every almost $p$-limited set is relatively weakly compact.  This question arises naturally from the fact that every $p$-limited set is relatively weakly compact (see \cite[Proposition 2.1]{delpin2}). In Theorem \ref{r2.3u}, we prove that a Banach lattice $E$ is a KB-space if and only if every almost $p$-limited (resp. positive $p$-limited) set in $E$ is relatively weakly compact. Finally, we prove, as an application of Theorem \ref{r2.3u}, that the adjoint of a weakly compact taking values on a Banach lattice operator is disjoint $p$-summing if and only if $E$ is a KB-space. A dual result is also obtained.


We refer the reader to \cite{Positive, Meyer} for background on Banach lattices and to \cite{albiac, fabian} for Banach space theory. Throughout this paper, $X, Y$ will denote Banach spaces and $E, F$ will denote Banach lattices. For a Banach space $X$, $B_X$ denotes its closed unit ball and $X^\ast$ denotes its topological dual. For a subset $A \subset E$, we denote by $\sol{A}$ its solid hull.   Throughout this article, we assume that $1 \le p <\infty$, unless otherwise stated. 
Now, we recall some useful definitions that will be needed in this paper.

\noindent (1) A subset $A$ of a Banach space $X$ is said to be \textit{Dunford-Pettis} if every weakly null sequence in $X^\ast$ converges uniformly to zero on $A$	 (see \cite[p. 350]{Positive}).

\noindent (2) A Banach space $X$ has the:
	 	 
         $\bullet$ \textit{Gelfand-Phillips (GP) property} if each limited set in $X$ is relatively compact (see \cite[Exercise 5.86]{fabian}).
         
	 	 	$\bullet$  \textit{Schur property} if each relatively weakly compact set in $X$ is relatively  compact (see \cite[Definition 2.3.4]{albiac}).
            
	 	 	$\bullet$ \textit{DP$_{rc}$P} if each Dunford-Pettis set in $X$ is relatively compact (see \cite{Emman DP}).

\noindent (3) A subset $A$ of a Banach lattice is said to be:

$\bullet$ \textit{almost Dunford-Pettis} (resp. \textit{almost limited}) if every disjoint weakly null sequence (resp. disjoint weak* null sequence) in $E^*$ converges uniformly to zero on $A$ (see  \cite{Almost DP set, Chen}). 


\noindent (4) A Banach lattice $E$ has the: 
	
     $\bullet$ \textit{Strong GP} if each almost limited set in $E$ is relatively compact (see \cite{Strong GP}).
		
		$\bullet$  \textit{Strong DP$_{rc}$P} if each almost Dunford-Pettis set in $E$ is relatively compact (see \cite{Strong DPrcP}).

\noindent (5)  An operator $T:X\rightarrow Y$ is called 

 $\bullet$ {\it completely continuous} if it takes weakly null sequences to norm null sequences (see \cite[p. 340]{Positive}). 

    $\bullet$  {\it $p$-summing} if it takes weakly $p$-summable sequences to $p$-summable sequences (see \cite{diestel}).
    


$\bullet$  {\it $p$-limited completely continuous} ($p$-lcc, in short) if for every $p$-limited weakly null sequence $(x_n)_n \subset X$, $\|Tx_n\|\rightarrow 0$ (see \cite{p-lcc}).

\noindent  (6) An operator $T:E\rightarrow X$ is called 

			$\bullet$ {\it order weakly compact} if it maps order intervals to relatively weakly compact ones (see \cite[p. 318]{Positive}).
            




            
	\section{The weak p-GP property} \label{Secweakpgp}

We recall from \cite{p-compact} that a bounded set $A \subset X$ is said to be \textit{relatively $p$-compact} if there
exists a norm $p$-summable sequence $(x_n)_n$ ($(x_n)_n \in c_0(X)$ if $p =\infty $) such that $$A \subseteq \left\{\sum\limits_{j=1}^\infty \lambda_j x_j : (\lambda_j)_j \in B_{\ell_{p^*}}\right\} = : p\mbox{-conv}\{(x_j)_j\}.$$
A Banach space $X$ has the \textit{GP property of order $p$ ($p$-GP property)} if each $p$-limited set in $X$ is relatively $p$-compact \cite{p-compact,delpin2}.
 Although this terminology was introduced in \cite{delpin2} by J. M. Delgando and C. Piñeiro, this notion was first addressed, with a different name, by the same authors in \cite{p-compact}. 
While reflexive spaces and separable dual spaces have the $p$-GP property, the Banach spaces $c_0$, $\ell_\infty$ and $L_1([0,1])$ fail to have the $p$-GP property (see \cite[Proposition ]{delpin2}).
Considering relatively compact sets instead of relatively $p$-compact ones, I. Ghenciu introduced and studied the $p$-limited relatively compact property (see \cite{p-lcc}). Since this is a weaker property than the $p$-GP introduced by Delgado and Piñeiro, we shall call it the weak $p$-GP property, that is a Banach space $X$ has the weak $p$-GP property if every $p$-limited subset of $X$ is relatively compact.
Since relatively $p$-compact sets are always relatively compact, the $p$-GP property implies the weak $p$-GP property, and it is still unknown if the weak $p$-GP property implies the $p$-GP property. 
  We divide this section in two parts. In the first subsection, we focus on studying the weak $p$-GP in Banach spaces, while in the second part we change our focus to the Banach lattice setting. In particular, we will provide conditions so that the weak $p$-GP and the $p$-GP property coincide in a Banach lattice.

\subsection{The weak p-GP property in Banach spaces}

We begin this subsection by proving a characterization of Banach spaces with the weak $p$-GP property in terms of $p$-lcc operators.

 \begin{proposition}\label{hSj} For a Banach space $X$, the following are equivalent:  \\
	{\rm (a)} $X$ has the weak $p$-GP property. \\
		{\rm (b)} For each Banach space $Y$, $\mathcal{L}^p_{lcc}(X,Y)=\mathcal{L}(X,Y)$. \\
		{\rm (c)} $\mathcal{L}^p_{lcc}(X,\ell_{\infty})=\mathcal{L}(X,\ell_{\infty})$.
\end{proposition}

\begin{proof} (a)$\Rightarrow$(b) Follows from \cite[p. 178, item (D)]{p-lcc}.

(b) $\Rightarrow$(c) Immediate.

	 $(c)\Rightarrow (a)$ If $X$ does not have the weak $p$-GP property, then there is a $p$-limited weakly null sequence
	$(x_n)_n\subset X$ with $\|x_n\| =1$ for all $n \in \N$ (see \cite[p. 178, item (D)]{p-lcc}).
	Choose a normalized sequence $(x_n^*)_n\subset X^*$ such that $x_n^*(x_n) = 1$ for every $n \in \N$, and define $T:X\rightarrow \ell_{\infty} $ by
	$Tx=(x_n^*(x))_n$ for every $x\in X$. Since $\displaystyle \|Tx_n\|_\infty = \sup_{k \in \N} |x_k^*(x_n)| \ge 1$ for every $n \in \N$, we get that $T$ is not $p$-lcc. 
\end{proof}


We will need the following result which is a vector-valued version of \cite[Proposition 1.2]{galmirams}.

\begin{theorem}\label{limitedness}
For a norm bounded subset $A \subset X$, the following conditions are equivalent:\\
{\rm (a)} $A$ is $p$-limited. \\
{\rm (b)} For every Banach space $Y$ and every $(T_n)_n\subset \mathcal{L}(X, Y)$ such that $\displaystyle \sum_{n=1}^{\infty}\|T_nx\|^p<\infty$ for every $x\in X$, it follows $\displaystyle \sum_{n=1}^{\infty}\left (\sup_{x\in A}\|T_nx\|\right )^p<\infty$.	\\
{\rm (c)} For every Banach space $Y$, every $(x_n)_n \subset A$ and every $(T_n)_n \subset \mathcal{L}(X, Y)$ such that $\displaystyle \sum_{n=1}^{\infty}\|T_nx\|^p<\infty$ for every $x\in X$, it follows $\displaystyle \sum_{n=1}^{\infty} \|T_nx_n\|^p<\infty$.
\end{theorem}
\begin{proof} 
	(a)$\Rightarrow$(b) Suppose that $A$ is a $p$-limited subset of $X$ and fix $n \in \N$. Since $T_n(A)$ is a bounded subset of $Y$, there exists $x_n \in A$  such that
    $\displaystyle\sup_{x\in A}\|T_nx\|\le 2\|T_nx_n\|.$ 
    By the Hahn-Banach theorem, exists $y_n^*\in Y^*$ with $\norma{y_n^*} = 1 $ and $y_n^*(T_nx_n)=\|T_nx_n\|.$ So, for an arbitrary
	$x\in X$, we have $$\sum_{n=1}^{\infty}\|T_n^*y_n^*(x)\|^p=\sum_{n=1}^{\infty}\|y_n^*(T_nx)\|^p\le \sum_{n=1}^{\infty}\|T_nx\|^p <\infty.$$
    Thus, $(T_n^*y_n^*(x))_n\in \ell_p$ for every $x \in X$, which implies that $(T_n^*y_n^*)_n$ is weakly $p$-summable sequence in $X^*$, and hence $(\displaystyle\sup_{x\in A}|T_n^*y_n^*(x)|)_n\in \ell_p$. Therefore,
    \begin{eqnarray*}
		\sum_{n=1}^{\infty} \left (\sup_{x\in A}\|T_nx\|\right)^p&\le& 2\sum_{n=1}^{\infty}\|T_nx_n\|^p \\
		&=& 2\sum_{n=1}^{\infty}|y_n^*(T_nx_n)|^p\\
		&\le& 2\sum_{n=1}^{\infty}\left (\sup_{x\in A}|y_n^*(T_nx)|\right )^p <\infty.
	\end{eqnarray*}
	
	\par (b)$\Rightarrow$(c) Obvious.

    (c)$\Rightarrow$(a) Take $Y = \mathbb{R}$ and apply \cite[Proposition 1.2]{galmirams}.
\end{proof}	

As an application of Theorem \ref{limitedness}, we will prove improvements of Theorem 5 and Theorem 8 in \cite{p-lcc}. First, we recall that the evaluation operators $\psi_{y^*}$ and $\phi_{x}$ on $\mathcal{M}\subset \mathcal{L}(X,Y)$ are defined by $\psi_{y^*} (T)=T^*y^*$ and $\phi_{x}(T)=Tx$, for all $y^* \in Y^* $, $x \in X$ and $T\in \mathcal{M}$.

\begin{theorem}\label{m2.1}
Suppose that $X^*$ has the weak $p$-GP property and that $Y$ is a Banach space. If $\mathcal{M}$ is a closed subspace of $\mathcal{L}(X, Y)$ contained in $\Pi_p(X, Y)$
such that all of the evaluation operators $\phi_{x}:\mathcal{M}\to Y$ are $p$-lcc, then $\mathcal{M}$ has the weak $p$-GP property. 
\end{theorem}

\begin{proof}	
Assuming that $X^*$ has the weak $p$-GP, we have that every $p$-summing operator from $X$ into $Y$ is compact (see \cite[Theorem 4]{Ioana}).
Now, let $H \subset \mathcal{M}$ be a $p$-limited set and we prove that $H$ is relatively compact. 
By \cite[Theorem 2.2]{ghenciucomp}, it suffices us to prove that that $H(x) = \conj{T(x)}{T \in H}$ is relatively compact in $Y$ for every $x \in X$ and that $H^*(B_{Y^*}) = \conj{T^*y^*}{T \in H, \, y^* \in B_{Y^*}}$ is relatively compact in $X^*$. On the one hand,
since $\phi_{x}:\mathcal{M}\to Y$ is $p$-lcc for every $x \in X$ and $H \subset \mathcal{M}$ is a $p$-limited set, $\phi_{x}(H)=H(x)$ is relatively compact for every $x \in X$. On the other hand, 
since $X^*$ has the weak $p$-GP, to prove that $H^*(B_{Y^*})$ is relatively compact in $X^*$ is enough to check that $H^*(B_{Y^*})$ is a $p$-limited set. To do this, let $(x_n)_n$ be a weakly $p$-summable sequence in $X$. Since $H \subset \mathcal{M} \subset \Pi_p(X, Y)$, we get that $$\displaystyle \sum_{n=1}^\infty \norma{\phi_{x_n}(T)}^p \leq \sum_{n=1}^\infty \norma{T(x_n)}^p < \infty \quad\text{for every}\quad T \in H.$$
So, it follows from Theorem \ref{limitedness} that 
\begin{align*}
    \sum_{n=1}^\infty \left (\sup_{x^* \in H^*(B_{Y^*})} |x^*(x_n)|
    \right)^p & = \sum_{n=1}^\infty  
    \left (\sup_{T\in H, y^*\in B_{Y^*}}| T^*y^*(x_n) |\right )^p \\
    &  = \sum_{n=1}^\infty \left
    ( \sup_{T\in H, y^*\in B_{Y^*}} | y^*(Tx_n) | \right )^p = \sum_{n=1}^\infty\left 
(\sup_{T\in H}\|T{x_n}\|\right )^p \\
& = \sum_{n=1}^\infty \left (\sup_{T \in H}\norma{\phi_{x_n}(T)}\right )^p < \infty
\end{align*}
proving that $H^*(B_{Y^*})$ is a $p$-limited set in $X^*$ by \cite[Corollary 3(ii)]{Ioana}, and we are done.
\end{proof}	

\par We say that an operator $T:X\to Y$ is $p$-limited if $T(B_X)$ is a $p$-limited subset of $Y$. The following theorem improves \cite[Theorem 5]{Ioana}. The class of all $p$-limited operators from $X$ to $Y$ is denoted by $\mathcal{L}_i^p(X,Y)$.

\begin{theorem}\label{gm2.1} Suppose that $Y$ has the weak $p$-GP property and that $X$ is a Banach space. If $\mathcal{M}$ is a closed subspace of $\mathcal{L}(X, Y)$ contained in $\mathcal{L}_i^p(X,Y)$
 such that all of the evaluation operators $\psi_{y^*}:\mathcal{M}\to Y$ are $p$-lcc, then $\mathcal{M}$ has the weak $p$-GP property.
\end{theorem}

\begin{proof}	
Assuming that $Y$ has the weak $p$-GP, it is easy to see that every $p$-limited operator $T: X \to Y$ is compact.
	Now, let $H \subset \mathcal{M}$ be a $p$-limited set and we prove that $H$ is relatively compact. By \cite[Theorem 2.2]{ghenciucomp}, it is enough to show that $H^*(y^*) = \conj{T^*y^*}{T \in H}$ is relatively compact in $X^*$ for every $y^* \in Y^*$ and that $H(B_{X}) = \conj{T(x)}{T\in H, \, x \in B_X}$ is relatively compact compact in $Y$. First, we notice that 
    since $\psi_{y^*}:\mathcal{M}\to Y$ is $p$-lcc for all $y^* \in Y^*$ and $H \subset \mathcal{M}$ is a $p$-limited set, $\psi_{y^*}(H)=H^*(y^*)$ is relatively compact for every $y^* \in Y^*$. Now, we observe that to prove that $H(B_{X})$ is relatively compact in $Y$, it is enough to prove that $H(B_{X})$ is $p$-limited. So, if $(y_n^*)_n$ be a weakly $p$-summable sequence in $Y^*$, we have from $H \subset \mathcal{M} \subset \mathcal{L}_i^p(X, Y)$ that $\displaystyle \sum_{n=1}^\infty \norma{T^*(y_n^*)}^p < \infty$ for every $T \in H$, because $T^*$ is $p$-summing (\cite[Theorem 3.1]{delpin2}).
    Now, since $H$ is a $p$-limited set in $\mathcal{M}$, it follows from Theorem \ref{limitedness} that 
    \begin{align*}
        \sum_{n=1}^\infty \left ( \sup_{y\in H(B_X)} |y_n^*(y)| \right )^p & = \sum_{n=1}^\infty \left ( \sup_{T\in H, x\in B_{X}} |y_n^*(T(x))| \right )^p \\
        & = \sum_{n=1}^\infty \left ( \sup_{T\in H, x\in B_{X}} |T^*y_n^*(x)| \right )^p = \sum_{n=1}^\infty \left ( \sup_{T\in H}\|T^*{y_n^*}\| \right )^p < \infty,
    \end{align*}
    proving that $H(B_{X})$ is $p$-limited and we are done.
\end{proof}	


The following corollary is an application of Theorems \ref{m2.1} and \ref{gm2.1}.

\begin{corollary}\label{ghm2.1} Let $\mathcal{M}$ be a closed subspace of $\mathcal{L}(X, Y)$ and 
suppose that both $Y$ and $X^*$ have the weak $p$-GP property. \\
{\rm (1)} If $\mathcal{M}\subset \mathcal{L}_i^p(X,Y)$, then $\mathcal{M}$ has the weak $p$-GP property. \\
{\rm (2)} If $\mathcal{M}\subset \Pi_p(X,Y)$, then $\mathcal{M}$ has the weak $p$-GP property.
\end{corollary}

	Now, we compare the weak $p$-GP property with the so-called coarse $p$-GP property.
We recall from \cite{coarse} that a bounded subset $A$ of a Banach space $X$ is said to be coarse $p$-limited if $T(A)$ is a relatively compact subset of $\ell_p$ for every bounded linear operator $T: X \to \ell_p$. A Banach space $X$ is said to have the coarse $p$-GP property whenever every coarse $p$-limited subset of $X$ is relatively compact. This property was introduced in by P. Galindo and the second author in \cite{coarse}. Then, J. Rodríguez and A. Rueda Zoca used this property weak precompact sets in projective tensor products of Banach spaces (see \cite{ruedazoca}). Since every $p$-limited subset of $X$ is coarse $p$-limited (see \cite[p. 943]{coarse}), it follows that if $X$ has the  coarse $p$-GP, then $X$ has the weak $p$-GP. The converse does not hold since reflexive spaces have the weak $p$-GP property, but there are reflexive spaces failing the coarse $p$-GP property (see \cite[Remark 4]{coarse}). We still do not know under which condition on a Banach space $X$ the weak $p$-GP property implies the coarse $p$-GP. So far, we have the following diagram:

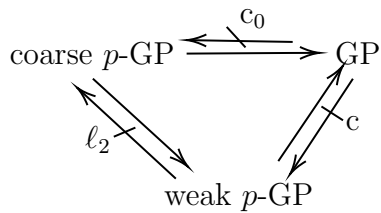
\begin{center}
	\tikzset{every picture/.style={line width=0.75pt}} 
	
	\begin{tikzpicture}[x=0.68pt,y=0.69pt,yscale=-1,xscale=1]
		
		\draw   (195,138) -- (252,137.65) -- (267,137.61) ;
		\draw [shift={(269,137.6)}, rotate = 179.83] [color={rgb, 255:red, 0; green, 0; blue, 0 }  ,draw opacity=1 ][line width=0.75]    (10.93,-3.29) .. controls (6.95,-1.4) and (3.31,-0.3) .. (0,0) .. controls (3.31,0.3) and (6.95,1.4) .. (10.93,3.29)   ;
		\draw [color={rgb, 255:red, 0; green, 0; blue, 0 }  ,draw opacity=1 ]   (142.85,151.75) -- (189.15,193.45) -- (192.15,196.45) -- (194.44,198.74) ;
		\draw [shift={(195.85,200.15)}, rotate = 225] [color={rgb, 255:red, 0; green, 0; blue, 0 }  ,draw opacity=1 ][line width=0.75]    (10.93,-3.29) .. controls (6.95,-1.4) and (3.31,-0.3) .. (0,0) .. controls (3.31,0.3) and (6.95,1.4) .. (10.93,3.29)   ;
		\draw    (154.5,183.6) -- (166.5,178.6) ;
		\draw   (287.5,151.8) -- (253.61,202.54) ;
		\draw [shift={(252.5,204.2)}, rotate = 303.74] [color={rgb, 255:red, 0; green, 0; blue, 0 }  ,draw opacity=1 ][line width=0.75]    (10.93,-3.29) .. controls (6.95,-1.4) and (3.31,-0.3) .. (0,0) .. controls (3.31,0.3) and (6.95,1.4) .. (10.93,3.29)   ;
		\draw   (270,167) -- (282,174.6) ;
		\draw   (254,131.6) -- (197,130.63) ;
		\draw [shift={(195,130.6)}, rotate = 0.97] [color={rgb, 255:red, 0; green, 0; blue, 0 }  ,draw opacity=1 ][line width=0.75]    (10.93,-3.29) .. controls (6.95,-1.4) and (3.31,-0.3) .. (0,0) .. controls (3.31,0.3) and (6.95,1.4) .. (10.93,3.29)   ;
		\draw   (217.3,123.36) -- (217.3,123.36) -- (224.5,131.1) -- (227.76,134.6) ;
		\draw   (246,197) -- (280.87,146.25) ;
		\draw [shift={(282,144.6)}, rotate = 124.49] [color={rgb, 255:red, 0; green, 0; blue, 0 }  ,draw opacity=1 ][line width=0.75]    (10.93,-3.29) .. controls (6.95,-1.4) and (3.31,-0.3) .. (0,0) .. controls (3.31,0.3) and (6.95,1.4) .. (10.93,3.29)   ;
		\draw [color={rgb, 255:red, 0; green, 0; blue, 0 }  ,draw opacity=1 ]   (189,206.6) -- (135.48,157.95) ;
		\draw [shift={(134,156.6)}, rotate = 42.27] [color={rgb, 255:red, 0; green, 0; blue, 0 }  ,draw opacity=1 ][line width=0.75]    (10.93,-3.29) .. controls (6.95,-1.4) and (3.31,-0.3) .. (0,0) .. controls (3.31,0.3) and (6.95,1.4) .. (10.93,3.29)   ;
		
		\draw (95,130) node [anchor=north west][inner sep=0.75pt]   [align=left] {coarse $p$-GP};
		\draw (276,130) node [anchor=north west][inner sep=0.75pt]   [align=left] {GP};
		\draw (180.85,207.15) node [anchor=north west][inner sep=0.75pt]   [align=left] {weak $p$-GP};
		\draw (282,171) node [anchor=north west][inner sep=0.75pt]  [color={rgb, 255:red, 0; green, 0; blue, 0 }  ,opacity=1 ] [align=left] {c};
		\draw (223,112) node [anchor=north west][inner sep=0.75pt]  [color={rgb, 255:red, 0; green, 0; blue, 0 }  ,opacity=1 ] [align=left] {c$\displaystyle _{0}$};
		\draw (137,176) node [anchor=north west][inner sep=0.75pt]  [color={rgb, 255:red, 0; green, 0; blue, 0 }  ,opacity=1 ] [align=left] {$\displaystyle \ell $$\displaystyle _{2}$};
	\end{tikzpicture}
\end{center}

\subsection{The weak p-GP property in Banach lattices}

The main purpose of this subsection is to develop a study of the weak $p$-GP property in Banach lattices. First, we observe that if a Banach lattice $E$ has the weak $p$-GP, then $E$ is a KB-space. Indeed,
since the canonical sequence $(e_n)_n$ is $p$-limited (for all $p \geq 1$) in $c_0$ (see \cite[Remark 2.2 and Proposition 2.1(3)]{delpin2}) which is not relatively compact, every Banach space with a copy of $c_0$ fails to have the weak $p$-GP property. Thus, $E$ must be a KB-space, and hence $E$ has the GP property (see \cite[Theorem 4.5]{WOrder}). On the other hand,  $c_0$ is a Banach space with the GP property that fails the weak $p$-GP. 
More can be said for Grothendieck Banach lattices:

\begin{theorem}\label{Sj} For a Grothendieck Banach lattice $E$, the following are equivalent:  \\
	{\rm (a)} $E$ has the $p$-GP property. \\
	{\rm (b)} $E$ has the weak $p$-GP property. \\
	{\rm (c)} $E$ has the GP property.
\end{theorem}

\begin{proof} 
    We already know that the
    implications (a)$\Rightarrow$(b) and (b)$\Rightarrow$(c) hold.
    It remains us to prove that  (c)$\Rightarrow$(a). Indeed, if $E$ is a Grothendieck Banach lattice with the GP property, then $E$ has the DP$_{rc}$P, which implies that $E$ fails to have a copy of $c_0$, because the canonical sequence $(e_n)_n$ is Dunford-Pettis in $c_0$ but not relatively compact. On the other hand, since $E$ is a Grothendieck Banach lattice, $E^*$ has order continuous norm (see \cite[Theorem 5.3.13]{Meyer}), and hence $E^*$ is a KB-space by \cite[Theorem 2.4.14]{Meyer}. Therefore, $E$ and $E^*$ are KB-spaces, which yield that $E$ is reflexive by \cite[Theorem 2.4.15]{Meyer}. It follows from \cite[Theorem 4.5]{delpin2} that $E$ has the $p$-GP property.
\end{proof}

	The following theorem shows that each discrete KB-space has the $p$-GP property, and hence the weak $p$-GP property.

\begin{theorem}\label{p-limited} 
	If $E$ is a discrete KB-space, then every $p$-limited subset of $E$ is relatively $p$-compact. Thus every discrete KB-space has the $p$-GP property, and hence the weak $p$-GP property.
\end{theorem}

\begin{proof}
	Let $A \subset E$ be a $p$-limited set. As in the proof of  implication (d)$\Rightarrow$(a) in \cite[Proposition 2.1]{p-compact}, to prove that $A$ is relatively $p$-compact, it suffices to show that each countable subset of $A$ is relatively $p$-compact. 
    So, let $B : = \conj{x_n}{n \in \mathbb{N}} \subset A$ and let $F$ be the separable closed sublattice of $E$ containing $B$ (see \cite[Ex. 9, p. 204]{Positive}). Considering the ideal $I(F)$ generated by $F$, we get from the fact that $F$ is separable and $E$ is a discrete that $G = \overline{I(F)}$ is also separable (see the proof of \cite[Theorem 6.5(i)]{buwong}). Thus, $G$ is a separable closed ideal of $E$, and consequently it is a separable discrete KB-space that is the range of a positive linear projection $P: E \to E$ (see \cite[Corollary 2.4.4]{Meyer}) and contains $B$. Moreover, it follows from \cite[Exercise 5.4.E2]{Meyer} that $G$ is the dual of a Banach lattice. 
	Finally, since $B \subset G$ is a $p$-limited subset of $E$ and $G$ is a complemented subspace of $E$, we get that $B$ is a $p$-limited subset of $G$ that is a separable dual space. By \cite[Proposition 3.6]{delpin2}, we obtain that $B$ is a relatively $p$-compact set in $G$, and so it is also relatively $p$-compact in $E$.
\end{proof}

It follows from Theorem \ref{p-limited} that all $\ell_p$, $1\le p<\infty$ and Banach lattices with the Schur property have the $p$-GP property. The converse of Theorem \ref{p-limited} is not true. Indeed, $L_2([0,1])$ is a non discrete Banach lattice with the $p$-GP (because it is reflexive). However, we can prove that whenever $E$ has the weak $p$-GP property and 
 $E^*$ is a weak $p$-consistent Banach lattice, then $E$ is a discrete KB-space. Following \cite[Definition 2.4]{zeefou}, 
a Banach lattice $E$ is \textit{weak $p$-consistent} if $(|x_n|)_n$ is weakly $p$-summable whenever $(x_n)_n$ is a weakly $p$-summable sequence in $E$.
 It follows from \cite[Propostion 2.7]{zeefou} that every \text{AM}-space with unit is a weak $p$-consistent Banach lattice. In our following result, we provide characterizations for weak $p$-consistent dual Banach lattices.

\begin{theorem}\label{2.6}  For a Banach lattice $E$, the
	following are equivalent: \\
    {\rm (a)} $E^*$ is a weak $p$-consistent Banach lattice. \\
	{\rm (b)} For every $x\in E^+$, the order interval $[-x,x]$ is a $p$-limited set. \\
	{\rm (c)} The solid hull of every $p$-limited  subset of $E$ is likewise a $p$-limited  set. \\
    {\rm (d)} Every bounded linear operator $T: E \to \ell_p$ is regular.  \\
    {\rm (e)} $E$ is isometrically isomorphic to an $AL$-space.
\end{theorem}
\begin{proof}
	(a)$\Leftrightarrow$(b) Let $(f_n)_n$ be a weakly $p$-summable sequence in $E^*$ and $x\in E^+$. The equivalence follows from equality
	$\sup_{y\in [-x,x]}|f_{n}(y)|=|f_{n}|(x)$.
	
	(a)$\Rightarrow$(c) Assume by way of contradiction that $A\subset E$ is a $p$-limited set such that $\sol{A}$ is not a $p$-limited. Then there exist a sequence
	$(x_n)_n$ in $\sol{A}$, a weakly $p$-summable sequence $(f_n)_n$ of $E^*$ such that $(f_n(x_n))_n\notin \ell_p$. For each $n \in \N$, there exists $y_n\in A$ such that $|x_n| \leq |y_n|$.  On the other hand, by \cite[Theorem 1.23]{Positive} there exists a sequence $(g_n)_n$ in $E^*$ with $|g_n| \leq |f_n|$ and $|f_n|(|y_n|)=g_n(y_n)$ for all $n$. 
	By hypothesis, the sequences $(|f_n|)_n$ and so $(g_n)_n$ are weakly $p$-summable. However
	$ |f_n(x_n)| \le |f_n|(|x_n|) \le |f_n|(|y_n|)=g_n(y_n)$. 
	Since $A$ is a $p$-limited subset of $E$, $(g_n(y_n))\in \ell_p$. Hence $(f_n(x_n))\in \ell_p$ which is a contradiction.

	The implication (c)$\Rightarrow$(b) is immediate.

(a)$\Rightarrow$(d) If $T: E \to \ell_p$ is a bounded linear operator given by $T(x) = (f_n(x))_n$, then $(f_n)_n$ is a weakly $p$-summable sequence in $E^*$, and so $(|f_n|)_n, (f_n^+)_n$ and $(f_n^-)_n$ are weakly $p$-summable sequences in  $E^*$. In particular, $R(x) = (f_n^+ (x))_n$ and $S(x) = (f_n^-(x))_n$ are positive operators such that $T = R - S$. Hence $T$ is a regular operator. 

(d)$\Rightarrow$(e) It follows from \cite[pg. 118]{busurv} that $\mathcal{L}^r(E, F)$ is isometrically isomorphic to $\mathcal{L}(E, F)$ if and only if $E$ is isometrically isomorphic to an $AL$-space or if $F^*$ is isometrically isomorphic to an $AM$-space. Since $\ell_p^*$ is not isometrically isomorphic to an $AM$-space, we get that $E$ is isometrically isomorphic to an $AL$-space.

(e)$\Rightarrow$(a) If $E$ is isometrically isomorphic to an $AL$-space, then $E^*$ is isometrically isomorphic to an $AM$-space with a unit, and hence $E^*$ is a weak $p$-consitent Banach lattice.
\end{proof}

The following is an easy consequence of the above theorem.

\begin{corollary}\label{yy2j.1} If $E^*$ is weak $p$-consistent and $E$ has the weak $p$-GP property, then $E$ is a discrete KB-space.
\end{corollary}

\begin{proof}
Since every Banach lattice with the weak $p$-GP is a KB-space, it remains us to prove that $E$ is discrete. In fact,
since $E^*$ is weak $p$-consistent, we get 
by Theorem \ref{2.6} that every order interval in $E$ is $p$-limited, and hence it is compact because $E$ has the weak $p$-GP.
It follows from \cite[Theorem 6.1]{WOrder} that $E$ is discrete, and we are done.
\end{proof}

As an application of Theorem \ref{2.6}, we provide a characterization for a dual Banach space having the weak $p$-GP.

\begin{corollary}\label{fShj} 
    A Banach space $X$ does not contain a copy of $\ell_1$ if and only if $X^*$ has the weak $p$-GP property.
\end{corollary}

\begin{proof}
	Assume first that $X$ does not contain a copy of $\ell_1$. If $A$ is a $p$-limited subset of $X^*$, then $A$ is an $L$-set in $X^*$ (apply \cite[Corollary 3(ii) and Proposition 1(iii)]{Ioana}), and so the assumption yieds that $A$ is relatively compact (see \cite[Theorem 2]{Emman RNP}, notice that there is a typo in the statement of this result, it should be \textit{``any (L) subset of $E^*$ is relatively compact.''}).

Now we assume that $X^*$ has the weak $p$-GP property. For the sake of contradiction, we suppose that $X^*$ contains a copy of $L_1([0,1])$, that is there exists a linear embedding $T: L_1([0,1]) \to X^*$. Since $L_1([0,1])$ is an $AL$-space, every order bounded sequence in $L_1([0,1])$ is $p$-limited by Theorem \ref{2.6}. Thus, the Rademacher sequence $(r_n)_n$ is a $p$-limited weakly null in $L_1([0,1])$, and so $(T(r_n))_n$ is a $p$-limited weakly null sequence in $X^*$. As $X^*$ has the weak $p$-GP property, $(T(r_n))_n$ is a norm null sequence in $X^*$ (see \cite[p. 178, item (D)]{p-lcc}), which is a contradiction with $\norma{r_n}_1 = 1$ for all $n \in \N$. Therefore, $X^*$ does not contain a copy of $L_1([0,1])$, and by \cite[p. 212]{Diestel}, $X$ does not contain a copy of $\ell_1$.  
\end{proof}

 \begin{corollary}  For a Banach lattice $E$, the following are equivalent:  \\
 {\rm (a)} $E$ is reflexive.\\
{\rm (b)} $E$ and $E^*$ have the $p$-GP property.\\
{\rm (c)} $E$ and $E^*$ have the weak $p$-GP property.
\end{corollary}

\begin{proof}
Since reflexive spaces have the $p$-GP property and the $p$-GP property implies the weak $p$-GP, it remains us to prove that (c)$\Rightarrow$(a). Indeed, assuming that both $E$ and $E^*$ have the weak $p$-GP, then $E$ and $E^*$ are KB-spaces, which imply that $E$ is reflexive (see \cite[Theorem 2.4.15]{Meyer}).
\end{proof}

As observed in the begging of this section, it is still unknown if the weak $p$-GP implies the $p$-GP. Considering the results proven in this section, if exists a Banach lattice $E$ with the weak $p$-GP failing the $p$-GP, then $E$ must be a non-discrete KB-space. Of course $E$ cannot be reflexive, and by combining Theorem \ref{2.6} and Corollary \ref{yy2j.1} we get that $E$ cannot be an AL-space. We conclude this section with  the following open question:

\medskip

\noindent \textbf{Question:} Does there exist a Banach space with the weak $p$-GP property, failing the $p$-GP property?


\section{The strong p-GP property} \label{secstrong}

Motivated by the definition of weak $p$-GP property, we use the class of almost $p$-limited sets to introduce the strong $p$-GP property in a Banach lattice, which is the strong version of the $p$-GP property.

\begin{definition}\label{2.1} A Banach lattice $E$ has the \textit{strong $p$-GP property} if every almost $p$-limited set in $E$ is relatively compact.
\end{definition}

\begin{remarks}
   {\rm (1)} Every Banach lattice with the strong $p$-GP has the weak $p$-GP, because every $p$-limited set in a Banach lattice is almost $p$-limited.  \\
    {\rm (2)} If $E$ is an $AM$-space with a unit, then $B_E$ is an almost $p$-limited subset of $E$ (see \cite[Remark 3.3]{galmirams}), which implies that $E$ fails to have the strong $p$-GP property. By the same reason, $c_0$ also fails to have the strong $p$-GP (see \cite[Remark 3.10]{galmirams}). \\
    {\rm (3)} Since order bounded sets are almost $p$-limited by \cite[Remark 3.3]{galmirams}, the Rademacher sequence $(r_n)_n$ is an almost $p$-limited set in $L_2([0,1])$ which is not relatively compact. Thus $L_2([0,1])$ is a Banach lattice with the $p$-GP property (hence the weak $p$-GP and the GP properties) failing to have the strong $p$-GP. Notice that the same conclusion holds for $L_p([0,1])$ with $1 < p < \infty$. 
   
\end{remarks}



In the following result, we present an important characterization of Banach lattices with the strong $p$-GP.

\begin{theorem}\label{u188} A Banach lattice $E$ has the strong $p$-GP property if and only if every weakly null almost $p$-limited sequence in $E$ is norm null.
\end{theorem}

\begin{proof}
We assume that every weakly null almost $p$-limited sequence in $E$ is norm null and we prove that $E$ has the strong $p$-GP property. First, we notice that since $(e_n)_n$ is a normalized $p$-limited weakly null sequence in $c_0$, $E$ cannot not have a copy of $c_0$, and so $E$ is a KB-space.  Let $A\subset E$ be an almost $p$-limited set. We show that $A$ is relatively compact. Let $(x_n^*)_n$ be an order bounded disjoint sequence in $E^*$. Then $(x_n^*)_n$ is weakly $p$-summable \cite{strong limited} and by the almost $p$-limitedness of $A$, $(\rho _A(x_n^*))_n=(\sup_{x\in A}|x_n^*|(|x|))_n\in \ell_p$. Since $E$ is a KB-space, it follows from \cite[Theorem 2.5.3]{Meyer} that $A$ is relatively weakly compact. Then each sequence $(x_n)_n$ in $A$ has a subsequence $(x_{n_k})$ weakly convergent to $x$.	Now, the sequence $(x_{n_k}-x)$ is weakly null and almost $p$-limited. By hypothesis (b), $\|x_{n_{k}}-x\|\rightarrow 0$ which implies that $A$ is relatively compact. Therefore,  $E$ has the strong $p$-GP property. The converse is immediate.
\end{proof}

It follows from Theorem \ref{u188} that every Banach lattice with the Schur property has the strong $p$-GP, and that every Banach lattice with the strong $p$-GP is a KB-space. We will prove in Theorem \ref{188} that $E$ is a discrete KB-space if and only if $E$ has the strong $p$-GP, whenever $p \geq 2$. First, we need the following interesting fact:

\begin{lemma}\label{example 3} {\rm (1)} Every $p$-limited subset of a Banach space is Dunford-Pettis.\\
{\rm (2)} Let $2 \leq p < \infty$. Every almost $p$-limited subset of a Banach lattice is almost Dunford-Pettis. 
\end{lemma}

\begin{proof}
(1) 	Let $X$ be a Banach space and let $A$ be a $p$-limited subset of $A$. It follows from \cite[Proposition 2.1]{delpin2} that the operator $U^*_A:X^* \rightarrow \ell_{\infty}(A)$ defined by $U^*_Ax^*(x)=x^*(x)$, where
$\ell_{\infty}(A)$ is the Banach space (with supremum norm) of all
bounded real-valued functions defined on $A$.
is a $p$-summing operator, and so 
it completely continuous (see \cite[Exercise 11.8]{floret}).
Hence for every weakly null sequence $(x_n^*)_n$ in $X^*$,  $\sup_{x\in A}|x_n^*(x)|=\|U^*_A(x_n^*)\|\to 0$ which implies that $A$ is a Dunford-Pettis set in $X$.

(2) Suppose that $p \geq 2$, let $E$ be a Banach lattice and let $A\subset E$ be an almost $p$-limited. By \cite[Proposition 3.5]{galmirams}, we may assume $A$ is solid. To prove that $A$ is almost Dunford-Pettis,  let $(x_n)_n$ be a disjoint sequence in $ A^+$ and let $(f_n)_n$ be a disjoint weakly null sequence of $E^*$. 
    By \cite[Proposition 3.8]{galmirams}, $(x_n)_n$ is $p$-limited, and by item (1) $(x_n)_n$ is a Dunford-Pettis sequence. Thus $f_n(x_n) \to 0$, which proves that $A$ is an almost Dunford-Pettis set (\cite[Corollary 2.5]{Almost DP set}).
\end{proof}

\begin{theorem}\label{188} 
For a Banach lattice $E$, consider the following conditions:\\
{\rm (a)} $E$ has the strong $p$-GP property.\\
{\rm (b)} $E$ is a discrete KB-space.\\
{\rm (c)} $E$ has the strong DP$_{rc}$P. \\
Thus, (a)$\Rightarrow$(b)$\Rightarrow$(c). Moreover, (c)$\Rightarrow$(a) whenever $p \geq 2$.
\end{theorem}

\begin{proof}  (a)$\Rightarrow$(b) Assuming that $E$ has the strong $p$-GP property, we already known that $E$ is a KB-space. Moreover, since every order interval in a Banach lattice is almost $p$-limited by \cite[Remark 3.3]{galmirams}, we get that every order interval in $E$ is  relatively compact, and hence $E$ is discrete (see \cite[Theorem 6.1]{WOrder}).
	\par (b)$\Leftrightarrow$(c) It follows from \cite[Theorem 2.3]{Strong DPrcP}. 
    
 Now, we suppose that $p\geq 2$ and prove that
 (c)$\Rightarrow$(a). Assume that $E$ has the strong DP$_{rc}$P and let $(x_n)_n$ be a weakly null almost $p$-limited sequence in $E$. By Lemma \ref{example 3}, $(x_n)_n$ is an almost Dunford-Pettis sequence, and by \cite[Theorem 2.9]{Strong DPrcP} we get that $\norma{x_{n}} \to 0$. Now, it follows from Theorem \ref{u188} that $E$ has the strong $p$-GP property.
\end{proof}

As a corollary, we get the following:

\begin{corollary} \label{corolario}
    For a Banach lattice $E$ and $p \geq 2$, the following are equivalent: \\
    (a) $E$ has the strong $p$-GP property. \\
    (b) Every weakly null positive $p$-limited sequence in $E$ is norm null. \\
    (c) Every positive $p$-limited subset of $E$ is relatively compact.
\end{corollary}

\begin{proof}
    The implication (a)$\Rightarrow$(b) follows from Theorem \ref{u188} and the fact that positive $p$-limited sets are almost $p$-limited (see \cite[Remark 3.10(5)]{galmirams}). The implication (b)$\Rightarrow$(c) can be obtained by reproducing the proof of Theorem \ref{u188} with the necessary adaptations. We prove that (c)$\Rightarrow$(a). So, we assume that every positive $p$-limited subset of $E$ is relatively compact. On the one hand, since $\conj{e_n}{n \in \N}$ is a positive normalized subset of $c_0$ that is not relatively compact, we get that $E$ must be a KB-space. On the other hand, since order intervals are positive $p$-limited, and by the assumption relatively compact, we get from \cite[Theorem 6.1]{WOrder} that $E$ is discrete. By Theorem \ref{188}, we conclude that $E$ has the strong $p$-GP property.
\end{proof}

\begin{corollary}\label{1i88} Let $p \geq 2$. A Banach lattice $E$ is discrete and reflexive if and only if $E$ and $E^*$ have the strong $p$-GP property.
\end{corollary}

\begin{proof} 
	 If $E$ and $E^*$ have the strong $p$-GP property, then by Theorem \ref{188}, $E$ and $E^*$ are discrete KB-space. Hence by \cite[Theorem 2.4.15]{Meyer}, $E$ is reflexive.
     
	Conversely, if $E$ is discrete and reflexive, by \cite[Corollary 2.6]{discrete} $E^*$  is discrete reflexive and by Theorem \ref{188}, $E$ and $E^*$ have the strong $p$-GP property. 
\end{proof}

\par By Theorem \ref{188} and  Theorem \ref{p-limited} we give the following corollary:

\begin{corollary}\label{1ih88} Let $p\geq 2$. For a discrete Banach lattice, the following are equivalent:\\
{\rm (a)} $E$ has the strong $p$-GP property.\\
{\rm(b)} $E$ has the weak $p$-GP property.\\
		{\rm (c)} $E$ has the $p$-GP property.
\end{corollary}


\par Our following result follows from Theorem \ref{Sj} and Corollary \ref{1ih88}:


\begin{corollary}\label{Sgj} Let $p\geq 2$. For a discrete Grothendieck Banach lattice $E$, the following are equivalent:  \\
{\rm (a)} $E$ has the strong $p$-GP property. \\
{\rm (b)} $E$ has the strong GP property. \\
{\rm (c)} $E$ has the weak $p$-GP property. \\
{\rm (d)} $E$ has the $p$-GP property. \\
{\rm (e)} $E$ has the GP property.
\end{corollary}

By considering a non-Grothendieck Banach lattice, Corollary \ref{Sgj} fails. For instance, $L_2([0,1])$ has the $p$-GP property, but it does not have the strong $p$-GP property.

\section{Miscellanea}

The final part of this article contains additional links of the notions we have discussed and related results.
While the strong $p$-GP property studies the Banach lattices in which every almost $p$-limited set is relatively compact, what can be told for Banach lattices in which every almost $p$-limited is relatively weakly compact?
It is known that $p$-limited sets are always relatively weakly compact (see \cite[Proposition 2.1]{delpin2}). This results does not remain true for almost $p$-limited sets. For instance, the closed unit ball of $C([0,1])$ is an almost $p$-limited set by \cite[Remark 3.3]{galmirams} which is not even weakly precompact. Moreover, we have the following:

\begin{theorem}\label{r2.3u} 
	For a Banach lattice $E$, the following are equivalent: \\
	{\rm (a)} Every almost $p$-limited set in $E$ is relatively weakly compact. \\
	{\rm (b)} Every positive $p$-limited set in $E$ is relatively weakly compact. \\
	{\rm(c)} $E$ is a KB-space. \\
	{\rm (d)} Every weakly Cauchy and almost $p$-limited sequence in $E$ is weakly convergent. \\
	{\rm (e)} Every weakly Cauchy and positive $p$-limited sequence in $E$ is weakly convergent. 
\end{theorem}

\begin{proof} 
(a)$\Rightarrow$(c) First note that $c_0$ has order continuous norm and $B_{c_0}$ is an almost $p$-limited subset of $c_0$ which is not relatively weakly compact. Thus,
	if each almost $p$-limited set in $E$ is relatively weakly compact, $E$ does not contain a copy of $c_0$, and so it is a KB-space.
	 	\par (c)$\Rightarrow$(a) Assume that $E$ is a KB-space and let $A$ be an almost $p$-limited subset of $E$. 
	 If $(x_n^*)_n$ is an order bounded disjoint sequence in $E^*$, then $(x_n^*)_n$ is weakly $p$-summable (see, e.g., \cite[P. 192]{Positive}), and by \cite[Proposition 2.2]{zeefou}, $(|x_n^*|)_n$ is also weakly $p$-summable.
	 Since $\sol{A}$ is almost $p$-limited, there exists $(a_j)_j \in \ell_p$ such that
	 $||x_n^*|(x)| \leq a_n$ for all $x \in \sol{A}$ and $n \in \N$, which implies that
	 $$ \rho_{A}(x_n^*) := \sup_{x \in A} |x_n^*|(|x|) \leq \sup_{y \in \sol{A}} ||x_n^*|(y)| \leq |a_n| \to 0. $$
	 By \cite[Theorem 2.5.3]{Meyer}, we obtain that $A$ is relatively weakly compact.

    The equivalence  (b)$\Leftrightarrow$(c) is similar to the one proved in (a)$\Leftrightarrow$(b), and the implication
    (c)$\Rightarrow$(d) is obvious.
    
\par (d)$\Rightarrow$(e) This follows because every positive $p$-limited set is almost $p$-limited by \cite[Remark 3.10(5)]{galmirams}.

\par (e)$\Rightarrow$(c) If $E$ is not a KB-space, there exists a positive embedding $T: c_0 \to E$. Considering $x_n = (1,\overset{n}{\dots}, 1, 0, \dots ) \in c_0$ for every $n \in \N$, we get that $(x_n)_n$ is weakly Cauchy and positive $p$-limited in $c_0$ (For the positive $p$-limited, \cite[Remark 3.10(7)]{galmirams}). Since $T$ is continuous, $(Tx_n)$ is weakly Cauchy in $E$, and since $T$ is positive, $(Tx_n)$ is positive $p$-limited in $E$. Now, the assumption yields that $(Tx_n)$ is weakly convergent in $E$. Let $z \in E$ such that $Tx_n \overset{w}{\longrightarrow} z$.
As $T(c_0)$ is a norm-closed subspace of $E$, $z \in T(c_0)$. This yields that $(x_n)_n$ is weakly convergent in $c_0$, a contradiction.
\end{proof}

	It is a natural question if there exists some type of weakly compact sets that contains the almost $p$-limited sets. 
In fact, every almost $p$-limited set is disjointly weakly compact whenever $p \geq 2$. We recall that a bounded subset $A$ of a Banach lattice $E$ is called a disjointly weakly compact set if for every disjoint sequence contained in $\sol{A}$ is weakly null (see, e.g., \cite[Definition 2.1]{disjointly}). The following implications, which hold for sets in Banach lattices whenever $p \geq 2$,
arise from Lemma \ref{example 3} and \cite[Remark 2.4(1)]{disjointly}:
$$ \text{almost $p$-limited} \Rightarrow  \text{almost Dunford-Pettis }\Rightarrow \text{disjointly weakly compact}.$$
The reverse arrows do not hold in general.  For instance, on the one hand, $B_{\ell_2}$ is a disjointly weakly compact subset of $\ell_2$ that is not almost Dunford-Pettis. On the other hand, the set $\displaystyle A = \left \{ \frac{e_n}{n^{1/p}} : n \in \N \right \} \subset \ell_p$ is relatively compact, and hence almost Dunford-Pettis, but it fails to be almost $p$-limited in $\ell_p$, because it is not order bounded.
	The following characterizes whenever every almost $p$-limited set in a Banach lattice is weakly precompact. 


\begin{corollary}\label{example 2}
Let $2\le p<\infty$ and $E$ be a Banach lattice. Then the following assertions are equivalent. \\
{\rm (a)} Every almost $p$-limited set in $E$ is weakly precompact. \\
{\rm (b)} Every order interval in $E$ is weakly precompact. \\
{\rm (c)} Every disjointly weakly compact set in $E$ is weakly precompact. 
\end{corollary}

\begin{proof}
    (a)$\Rightarrow$(b) This follows from the fact that order intervals are always almost $p$-limited sets proven in \cite[Remark 3.3]{galmirams}. 

    (b)$\Leftrightarrow$(c) This equivalence is established in \cite[Corollary 3.11]{disjointly}.
    
    (c)$\Rightarrow$(a) This follows from the already observed fact that every almost $p$-limited set is disjointly weakly compact.
\end{proof}

	By Corollary \ref{example 2} and \cite[Lemma 3.12]{disjointly}, whenever $p \geq 2$, a $\sigma$-Dedekind complete Banach lattice $E$ has order continuous norm if and only if every almost $p$-limited set in $E$ is weakly precompact.
It is important to notice that Lemma \ref{example 3}(2), Theorem \ref{188}, Corollary \ref{1i88}, Corollary \ref{1ih88}, Corollary \ref{Sgj}, and Corollary \ref{example 2} hold for $p \geq 2$, because they all depend on \cite[Proposition 3.8]{galmirams}. 

In \cite{delpin2}, Delgado and Piñeiro established important relations between the class of $p$-summing operators and the class of $p$-limited. Similar relations can be presented to the classes of disjoint/positive $p$-summing operators and the classes of almost/positive $p$-limited sets.  Following \cite{blasco} (resp. \cite{chenbelacel}), a bounded linear operator $T: E \to Y$ is said to be \textit{positive $p$-summing} (resp. \textit{disjoint $p$-summing}) if there exists a constant $C > 0$ such that for all  positive elements (resp. pairwise disjoint elements) $x_1, \dots, x_k \in E$, we have
$$ \left ( \sum_{i=1}^k \norma{Tx_i}^p \right )^{1/p} \leq C \sup \conj{\left ( \sum_{i=1}^\infty |x'(x_i)|^p \right )^{1/p}}{ x' \in B_{E'} }. $$
We refer to \cite{blasco} and \cite{chenbelacel} for the respective definitions and some results concerning these two classes of operators. 

The following result is standard and it establishes important relations between the classes mentioned above.

\begin{proposition} \label{psum} For a bounded linear operator $T: E \to X$, the following are equivalent:\\
{\rm (a)} $T$ is positive $p$-summing (resp. disjoint $p$-summing). \\
{\rm (b)} $T$ maps positive weakly $p$-summing (resp. disjoint weakly $p$-summing) sequences in $E$ to norm $p$-summable sequences in $F$. \\
{\rm (c)} $T^*(B_{X^*})$ is a positive $p$-limited (resp. an almost $p$-limited) subset of $E^*$.
\end{proposition}

We also have a dual of equivalence (a)$\Leftrightarrow$(c) in Propositon \ref{psum}:

\begin{proposition} \label{psum2}
The adjoint operator $T^*: E^* \to X^*$ is positive $p$-summing (resp. disjoint $p$-summing) if and only if  $T(B_X)$ is positive $p$-limited (resp. almost $p$-limited) subset of $E$. 
\end{proposition}

\begin{proof}    Assume that $T^*: E^* \to X^*$ is positive $p$-summing and let $(x_n^*)_n$ be a positive weakly $p$-summable sequence in $E^*$. Thus,     $$ \sum_{n=1}^\infty(\sup_{x \in B_X} |x_n^*(Tx)|)^p = \sum_{n=1}^\infty (\sup_{x \in B_X} |(T^*x_n^*)(x)|)^p = \norma{T^*x_n^*}^p < \infty, $$   which proves that $T(B_X)$ is a positive $p$-limited set in $E$. The converse holds tracing back the above argument.\end{proof}

\begin{remarks}
    (1) If $E$ is an $AL$-space, then $E^*$ is an $AM$-space with an order unit, and so $B_{E^*}$ is both a positive $p$-limited set and an almost $p$-limited set (see \cite{galmirams}). By Theorem \ref{psum}, we conclude that the identity operator $I_E$ is both a positive $p$-summing operator and a disjoint $p$-summing operator. This yields that every positive (resp. disjoint) weakly $p$-summable sequence in $E$ is norm $p$-summable, which is an interesting contrast with the weak Dvoretzky-Rogers Theorem \cite[2.18]{diestel}.  \\
    (2) It is a well-known fact that every $p$-summing operator between two Banach spaces is weakly compact. This is not true for disjoint $p$-summing or positive $p$-summing operators. In fact, as we have already pointed out the identity operator $Id_{\ell_1}$ is both disjoint $p$-summing and positive $p$-summing operator which is not weakly compact. Actually, we have that every disjoint $p$-summing 
operator is order weakly compact. Indeed, if $T: E \to F$ is disjoint $p$-summing and $(x_n)_n$ is an order bounded disjoint sequence in $E$, we have from \cite[Lemma 2.10]{strong limited} that $(x_n)_n$ is weakly $p$-summable, and hence $\norma{Tx_n} \to 0$, proving that $T$ is order weakly compact (see \cite[Theorem 5.57]{Positive}). \\
{\rm (3)} A version of \cite[Proposition 2.1(1)]{delpin2} can also be obtained: a norm bounded subset $A$ of a Banach lattice $E$ is disjoint $p$-summing (resp. positive $p$-summing) if and only if $U_A^*: E^* \to \ell_\infty(A)$ is disjoint $p$-summing (resp. positive $p$-summing).
\end{remarks}

It is well known that $p$-summing operators are weakly compact. We conclude this manuscript with a result that addresses this problem:

\begin{theorem} \label{ppsum} 
{\rm (1)} A Banach lattice $E$ is a KB-space if and only if, for every Banach space $X$, the bounded linear operator $T: X \to E$ is weakly compact, whenever $T^*$ is disjoint $p$-summing or positive $p$-summing. \\
{\rm (2)} The dual $E^*$, of a Banach lattice $E$, is a KB-space if and only if every disjoint $p$-summing (resp. positive $p$-summing) operator from $E$ into any Banach space is weakly compact. \\
{\rm (3)} A Banach lattice $E$ has the strong $p$-GP property if and only if, for every Banach space $X$, the bounded linear operator $T: X \to E$ is compact, whenever $T^*$ is disjoint $p$-summing. For $p\geq 2$, we may also consider positive $p$-summing operators by Corollary \ref{corolario}.\\
{\rm (4)}  The dual $E^*$, of a Banach lattice $E$, has the strong $p$-GP property if and only if every disjoint $p$-summing operator from $E$ into any Banach space is weakly compact. For $p\geq 2$, we may also consider positive $p$-summing operators by Corollary \ref{corolario}.
\end{theorem}

\begin{proof} 
(1) Suppose first that $E$ is a KB-space. If $T: X \to E$ is a bounded linear operator whose adjoint $T^*: E^* \to F^*$ is disjoint $p$-summing (resp. positive $p$-summing), then $T(B_X)$ is an almost $p$-limited (resp. positive $p$-summing) subset of $E$ by Proposition \ref{psum2}. Since $E$ is a KB-space, we have from Theorem \ref{r2.3u} that $T(B_X)$ is a relatively weakly compact subset of $E$, proving that $T$ is weakly compact. For the converse, we prove that every almost $p$-limited subset $A$ of $E$ is relatively weakly compact. In fact, 
let $(x_n)_n$ be a sequence in $A$ and defining $T:
	\ell_{1}\rightarrow E $ by
	$T(a)=\displaystyle \sum_{i=1}^\infty a_i x_i$ for every $a = (a_i)_i \in \ell_1$. We claim that the adjoint operator 
    $T^*:
	 E^* \rightarrow \ell_{\infty}$, that is given by $T^*(x^*)=(x^*(x_i) )_i$ for every $x^* \in X^*$, is disjoint $p$-summing. Indeed, if $(x_n^*)_n$ is a disjoint weakly $p$-summable sequence in $E^*$, we have
     $(\|T^*x_n^*\|)_n=(\sup_{i} |x_{n}^*(x_i)|)_n\in \ell_p$, because $(x_n)_n$ is an almost $p$-limited sequence. Therefore, $T^*$ is disjoint $p$-summing, and then $T$ is weakly compact. Thus, $(T (e_n^*))_n = (x_n)_n$ has a weakly convergent subsequence, which proves that $A$ is relatively weakly compact. By Theorem \ref{r2.3u}, we conclude that $E$ is a KB-space.

(2) Assume that $E^*$ is a KB-space. Let $X$ be a Banach space and let $T : E \to X$ be a disjoint $p$-summing (resp. positive $p$-summing) operator. By Proposition \ref{psum}, $T^* (B_{X^*})$ is an almost $p$-limited (positive $p$-limited) set in $E^*$, and by Theorem Theorem \ref{r2.3u}, we get that $T^* (B_{X^*})$ is relatively weakly compact (resp. relatively compact), which implies that $T^*$, and so $T$, is weakly compact. For the converse, let $A$ be an almost $p$-limited set in $E^*$ and let $(x_n^*)_n$ be a sequence in $A$. It is not difficult to see that the operator $T:E \rightarrow \ell_{\infty}$ defined by $T(x)=(x_i^*(x) )_i$ is
disjoint $p$-summing, and so the assumption yields that $T$ is weakly compact. By the Gantmacher`s theorem, $T^*$ is also a weakly compact operator, which yields that $(T ^*(e_n^*))_n = (x_n^*)_n$ has a
	weakly convergent subsequence. This proves that every almost $p$-limited subset of $E^*$ is relatively weakly compact, and hence $E^*$ is a KB-space.

The items (3) and (4) follow from the same arguments used to prove (1) and (2) by making the necessary changes.
\end{proof}

\bibliographystyle{amsplain}

\end{document}